\documentclass[10pt]{article}
\usepackage{amsthm,amsfonts,amsmath,amscd,amssymb,epsfig}
\usepackage[arrow, matrix, curve]{xy}
\theoremstyle{plain}
\newtheorem{theorem}{Theorem}[section]

\newtheorem{corollary}[theorem]{Corollary}

\newtheorem{lemma}[theorem]{Lemma}

\theoremstyle{remark}

\newcommand{\N}{{\mathbb{N}}}
\newcommand{\R}{{\mathbb{R}}}
\newcommand{\C}{{\mathbb{C}}}

\renewcommand{\H}{{\mathbb{H}}}

\title{Gromov Convergence in H\"older Spaces and Minimal Connections on Jet Bundles\\ }

\begin{document}

\vspace{2ex}
\begin{center}
\bfseries\Large{Gromov Compactness in H\"older Spaces and Minimal Connections on Jet Bundles}
\end{center}
\begin{center}
\large{Viktor Fromm}
\end{center}
\vspace{2ex}
The goal of this work is to establish a proof of the Gromov compactness in H\"older spaces
for curves with a totally real boundary condition following the original
geometric idea in \cite{Gr}. We use a local reflection principle in neighbourhoods of the totally
real submanifold as developed in \cite{Iv2} and existence results for special connections on spaces of jet bundles to obtain higher regularity
and Gromov-Schwarz estimates along the boundary. \\ \\
This work is based on my diploma thesis \cite{Fro} written at the Humboldt University Berlin. I would like to express my heartfelt gratitude to my advisor Klaus Mohnke for his help and encouragement. I am also deeply indepted to Wilhelm Klingenberg for his advice and for many fruitful discussions.
\section{Introduction and Statement of the Main Result}
We wish to give a proof of the following compactness theorem
\begin{theorem}
Let $X^{n}$ and $M^{2n}$ be compact manifolds of class $\mathcal{C}^{k+1,\alpha}$, $\{J_{n}\}_{n \in \N}$ a sequence of $\mathcal{C}^{k,\alpha}$-differentiable almost complex structures on $M$ that converges uniformly in $\mathcal{C}^{k,\alpha}(M)$ to an almost complex strcuture $J_{\infty}$ and $$\{\iota_{n}\}_{n \in \N} \subset \mathcal{C}^{k+1,\alpha}(X,M)$$ a sequence of embeddings converging uniformly in $\mathcal{C}^{k+1,\alpha}(X,M)$ to a $J_{\infty}$-totally real embedding $\iota_{\infty}$. Let $\Sigma$ be a compact smooth surface with boundary, $\{j_{n}\}_{n \in \N}$ complex strucures on $\Sigma$ and $$\{u_{n}\}_{n \in \N} \subset \mathcal{C}^{0} \cap L^{1,2}_{loc}(\Sigma,M)$$ a sequence of $j_{n}$-$J_{n}$-pseudoholomorphic curves with $u_{n}(\partial \Sigma) \subset \iota_{n}(X)$. \\ \\
If the energies $$\{\mathcal{E}(u_{n})=\frac{1}{2} \|du_{n}\|^{2}_{L^{2}(\Sigma \setminus \partial \Sigma, M)}\}_{n \in \N}$$ are uniformly bounded above by some constant $C_{\mathcal{E}}$ then there exist a regular collection $\Delta$ of curves on $\Sigma$, diffeomorphisms $\{\varphi_{n}\}_{n \in \N} \subset \mathcal{C}^{\infty}(\Sigma,\Sigma)$, a complex structure $j_{\infty}$ on $\Sigma \setminus \Delta$ such that components of $\Delta$ correspond to punctures on $\Sigma \setminus \Delta$ and a mapping $$u_{\infty} \in \mathcal{C}^{0}(\Sigma,M) \cap L^{1,2}_{loc}(\Sigma \setminus 
 \Delta,M)$$ with $u_{\infty}(\partial \Sigma) \subset \iota_{\infty}(X)$ such that the following holds
\begin{enumerate}
\item (Convergence of Complex Structures) $(\varphi^{*}_{n})j_{n} \rightarrow j_{\infty}$ uniformly in $\mathcal{C}^{\infty}(\Sigma \setminus \Delta)$.
\item (Properties of the Limit) The restriction of $u_{\infty}$ to any component of $\Delta$ is constant
and the restriction to $\Sigma \setminus \Delta$ is $j_{\infty}$-$J_{\infty}$-pseudoholomorphic.
\item (Convergence of Energies) $\mathcal{E}(u_{n}) \rightarrow \mathcal{E}(u_{\infty})$
\item (Regularity and Convergence of Curves) We have $u_{\infty} \in \mathcal{C}^{k+1,\alpha}(\Sigma \setminus \Delta,M)$ and $\mathcal{C}^{k+1,\alpha'}$-convergence $u_{n} \circ \varphi_{n} \rightarrow u_{\infty}$ holds uniformly on compact subsets of $\Sigma \setminus \Delta$ for each $\alpha' \in (0,\alpha)$, furthermore $u_{n} \circ \varphi_{n} \rightarrow u_{\infty}$ uniformly in $\mathcal{C}^{0}(\Sigma)$.  
\end{enumerate}

\end{theorem}
Here $k$ is any non-negative integer and $\alpha$ lies in the interval $(0,1)$. Thus we obtain Gromov convergence in the spaces $\mathcal{C}^{k+1,\alpha'}$ without any additional assumptions on the almost complex structure. \\ \\ Let us notice that an idea of proof of the theorem in the case of almost complex structures that are integrable in some neighbourhood of the totally real submanifold is mentioned in \cite{Gr}, paragraph 2.1.D and further explained in \cite{Pa}, paragraph 3.4 whereas a proof of the result in the context of Sobolev spaces $L^{1,p}$, $p>2$ is given in \cite{Iv}. 
\section{Regularity and Gromov-Schwarz estimates}
In this section we develop the analytic setup for the construction of the limit map $u_{\infty}$ in theorem 1.1. Let us first recall some properties of the Cauchy-Riemann operator on $\C$:
\begin{theorem}
Let $\overline{D} \subset \C$ denote the closed unit disk, $J_{st}$ the standard complex structure on $\R^{2n}$, given by $(v,w) \mapsto (-w,v)$ and $\overline{\partial}: L^{1,2}(\overline{D},\R^{2n}) \rightarrow L^{2}(\overline{D},\R^{2n})$ resp. $\partial: L^{1,2}(\overline{D},\R^{2n}) \rightarrow L^{2}(\overline{D},\R^{2n})$ be given by $$f \mapsto \frac{1}{2}(\frac{\partial f}{ \partial x} + J_{st} \frac{\partial f}{ \partial y})$$ resp. by $$f \mapsto \frac{1}{2}(\frac{\partial f}{ \partial x} - J_{st} \frac{\partial f}{ \partial y})$$ There exist operators $$P: L^{2}(\overline{D},\R^{2n}) \rightarrow L^{1,2}(\overline{D},\R^{2n})$$ and and $$T: L^{2}(\overline{D},\R^{2n}) \rightarrow L^{2}(\overline{D},\R^{2n})$$ such that the following holds:
\begin{enumerate}
\item $(P \circ \overline{\partial})|_{L^{1,2}_{0}(\overline{D},\R^{2n})}=Id_{L^{1,2}_{0}(\overline{D},\R^{2n})}$ 
\item $(T \circ \overline{\partial})|_{L^{1,2}_{0}(\overline{D},\R^{2n})}=\partial|_{L^{1,2}_{0}(\overline{D},\R^{2n})}$.
\item For each $p>2$ the restriction of $T$ to $L^{p}(\overline{D},\R^{2n})$ is a bounded operator $L^{p}(\overline{D},\R^{2n}) \rightarrow L^{p}(\overline{D},\R^{2n})$. 
\item $P$ restricts to a bounded operator $L^{p}(\overline{D},\R^{2n}) \rightarrow \mathcal{C}^{1-2/p}(\overline{D},\R^{2n})$ for every $p>2$. 
\item For every $\alpha \in (0,1)$ the restrictions of $P$ and of $T$ to $\mathcal{C}^{\alpha}(\overline{D},\R^{2n})$ are bounded operators $P: \mathcal{C}^{\alpha}(\overline{D},\R^{2n}) \rightarrow \mathcal{C}^{1,\alpha}(\overline{D},\R^{2n})$ and $T: \mathcal{C}^{\alpha}(\overline{D},\R^{2n}) \rightarrow \mathcal{C}^{\alpha}(\overline{D},\R^{2n})$.
\item $P$ and $T$ restrict to bounded operators $L^{1,p}(\overline{D},\R^{2n}) \rightarrow L^{2,p}(\overline{D},\R^{2n})$ resp. $L^{1,p}(\overline{D},\R^{2n}) \rightarrow L^{1,p}(\overline{D},\R^{2n})$ for every $p>1$.
\end{enumerate}
\end{theorem}
\begin{proof}
See \cite{Ve}. Part 1 and part 2 follow from theorem 1.17 on page 29, part 4 is theorem 1.20 on page 32, part 5 is a consequence of theorem 1.33 on page 49 and parts 3 and 6 follow from theorem 1.27 on page 40.
\end{proof}
We recall that in the case of curves without boundary Gromov-Schwarz estimates for higher derivatives can be obtained using the operators $P$ and $T$ and a bootstrap argument:
\begin{theorem}
Let $M$ be a closed manifold with some fixed Riemannian metric, $J$ a $\mathcal{C}^{k,\alpha}$-smooth almost complex structure on $M$, $\mu \in(0,1)$ and $S$ a hyperbolic surface without boundary. There exist positive constants $\delta=\delta(J)$ and $C=C(J,\mu)$ such that for any nonconstant pseudoholomorphic map $u \in \mathcal{C}^{0} \cap L^{1,2}_{loc} (S,M)$, any $z \in S$ and any $r>0$ such that the image of the metric ball $B_{r}(z) \subset S$ satisfies $ \text{diam}(u(B_{r}(z)))<\delta$ the following estimate holds: 
$$\|u \|_{\mathcal{C}^{k+1,\alpha}(B_{\mu r}(z),M)}<C r^{-(k+1+\alpha)} \text{diam} (u(B_{r}(z)))$$
In particular $u$ is differentiable of class $\mathcal{C}^{k+1,\alpha}$. \\ \\
In addition, the constant $C$ can be chosen to depend continuosly on $J$ in the following sense: If convergence $J_{n} \rightarrow J$ holds in $\mathcal{C}^{k,\alpha}(M)$ then we have $$C_{n}=C(J_{n},\mu) \rightarrow C=C(J,\mu)$$
\end{theorem}
\begin{proof}
 See \cite{Si}, theorem 2.3.6 on page 171.
\end{proof}
Let us turn to the case of curves with boundary on a totally real submanifold. Here the idea is to conduct a local doubling construction on the complex vector bundle $u^{*}TM$. However, in order to obtain Gromov-Schwarz estimates for higher derivatives we need to overcome the technical difficulty that the double of the almost complex structure $J$ will in general not even be of class $\mathcal{C}^{1}$. We do this by successively applying the doubling contruction to almost complex structures on spaces of jets of pseudoholomorphic maps rather than directly to $u^{*}J$. The choice of these almost complex structures is not unique and we will have to keep track of their regularity.\\ \\
The following theorem is our main technical result. The first steps of the proof are similar to arguments in paragraph 2.4 of \cite{Si}. A somewhat weaker statement using regularity theory of the Laplace operator and leading to estimates for Sobolev norms can be found as Proposition B.4.9 in \cite{McSal2}. 
\begin{theorem}
Under the assumptions of theorem 2.1 let $X^{n} \subset M^{2n}$ be a closed $J$-totally real submanifold, $S$ a hyperbolic surface with boundary and $u$ satisfy the boundary condition $u(\partial S) \subset X$. Then there are constants $\delta_{\partial}=\delta_{\partial}(J,X)$ and $C_{\partial}=C_{\partial}(J,X,\mu)$ such that the estimate of theorem 2.1 holds around any point $z \in \partial S$.
\end{theorem}
\begin{proof} We divide the proof into several steps.\\ \\ 
$Step$ 1. We introduce suitable coordinates in a neighbourhood of $X$. \\ \\
Fix $\varepsilon>0$. Let $\{U_{i}\}_{i=1,...,n}$ be a cover of $X$ by open subsets of $M$ such that the follwing holds:
\begin{enumerate}
\item For every $i$ there is an open neighbourhood $\overline{U}_{i} \subset U'_{i} \subset M$ and a diffeomorphism $$\Phi_{i}:  \R^{2n} \supset V_{i} \rightarrow  U'_{i}    $$ of class $\mathcal{C}^{k+1,\alpha}$ with $\Phi_{i}(X \cap U_{i}) \subset \R^{n}$ and $(\Phi^{*}_{i}J-J_{st})|_{\R^{n} \cap V_{i}}$ such that the estimates $$\| \Phi^{*}_{i}J -J_{st} \|_{\mathcal{C}^{k,\alpha}(\Phi^{-1}_{i}(U_{i}))}<\varepsilon$$ $$\|\Phi_{i}\|_{\mathcal{C}^{k+1,\alpha}(\Phi^{-1}_{i}(U_{i}),U_{i})}<C'$$ and $$\|\Phi^{-1}_{i}\|_{\mathcal{C}^{k+1,\alpha}(U_{i},\Phi^{-1}_{i}(U_{i}))}<C' $$ hold for every $i$ with some consant $C'=C'(J,X,\varepsilon)$. 
\item If $\tau: \R^{2n} \rightarrow \R^{2n}$ denotes the standard involution $(v,w) \mapsto (v,-w)$ then $\tau(V_{i})=V_{i}$.
\item There is $\delta_{\partial}=\delta_{\partial}(J,X)>0$ such that for any $x \in X$ the inclusion $B_{\delta_{\partial}}(x) \subset U_{i}$ holds for some $i$.  
\end{enumerate}
$Step$ 2. We show that for each $p>2$ the mapping $u$ is of class $\mathcal{C}^{1-\frac{2}{p}}$ up to the boundary and the estimate
$$ \|u \|_{\mathcal{C}^{1-\frac{2}{p}}(B_{\mu r}(z),M)}<C_{0} r^{-1+\frac{2}{p}} \text{diam} (u(B_{r}(z)))$$ holds with some constant $C_{0}=C_{0}(J,X,\mu,p)$. Here $B_{r}(z)$ is the metric ball in $S$ around $z \in \partial S$.\\ \\ 
We choose a biholomorphism $\varphi: \overline{D}_{+} \rightarrow \overline{B}_{r}(z)$, fix $\mu' \in (0,1)$ with $\varphi(\mu'\overline{D}_{+}) \subset  \overline{B}_{\mu r}(z)$ with $\|\varphi \|_{\mathcal{C}^{1}}$ and $\|\varphi^{-1} \|_{\mathcal{C}^{1}}$ bounded by some constant multiple of $r$. Consider the composition $$\tilde{u}:= \Phi^{-1}_{i} \circ u \circ \varphi : \overline{D}_{+} \rightarrow \R^{2n}$$ where $\overline{D}_{+}=\overline{D} \cap \{z \in \C: \text{ Im}(z) \geq 0 \}$. We note that with $J_{i}:=\Phi^{*}_{i}J$ the map $\tilde{u}$ satisfies the equation $$\overline{\partial}_{J_{i} \circ \tilde{u}}\tilde{u}:=\frac{1}{2}(\frac{\partial \tilde{u}}{ \partial x}+(J_{i} \circ \tilde{u})\frac{\partial \tilde{u}}{\partial y})=0$$ in the interior of $\overline{D}_{+}$, furthermore the boundary condition $\tilde{u}(\overline{D}_{+} \cap \R) \subset \R^{n}$ holds. Following \cite{Iv2} we define an extension of $\tilde{u}$ to a map $\tilde{u}^{d}: \overline{D} \rightarrow \R^{2n}$ by setting $$\tilde{u}^{d}(z)=\text{ext}(\tilde{u})(z):=(\tau \circ \tilde{u})(\tau^{d}z)$$ for $z \in \overline{D}_{-}$. Here $\tau^{d}: \C \rightarrow \C$ denotes the canonical involution $z \mapsto \overline{z}$. Let us prove $\tilde{u}^{d} \in L^{1,2}_{loc}(D,\R^{2n})$. To this effect let $\sigma \in \mathcal{C}^{\infty}_{0}(\overline{D},\R)$ be a test function. Using the fact that according to theorem 2.2 we have $\tilde{u}^{d} \in \mathcal{C}^{0}(\overline{D},\R^{2n}) \cap \mathcal{C}^{1}(\overline{D} \setminus \overline{D} \cap \R,\R^{2n})$ and applying Stokes theorem we compute $$ \int_{z' \in D} (\tilde{u}^{d} d \sigma)(z') dz' \wedge d\overline{z}'$$ $$=\int_{z' \in \overline{D}_{+}} (\tilde{u} d \sigma)(z') dz' \wedge d\overline{z}'+\int_{z' \in \overline{D}_{-}} (\tilde{u}^{d} d \sigma)(z') dz' \wedge d\overline{z}'$$ 
$$=\int_{x \in D_{+} \cap \R } (\sigma \tilde{u})(x) dx -\int_{z' \in D_{+} \setminus \R } (\sigma d\tilde{u})(z') dz' \wedge d\overline{z}'$$ $$+\int_{z'' \in D_{-}} (\tau \circ \tilde{u})(\tau^{d}(z'')) (d \sigma)(z'') dz'' \wedge d\overline{z}''$$ $$= \int_{x \in D_{+} \cap \R } (\sigma \tilde{u})(x) dx -\int_{z' \in D_{+} \setminus \R } (\sigma d\tilde{u})(z') dz' \wedge d\overline{z}'$$ $$- \int_{x \in D_{+} \cap \R } (\sigma (\tau \circ \tilde{u}))(x) dx-\int_{z'' \in D_{-} \setminus \R} \sigma(z'') d(\tau \circ \tilde{u} \circ \tau^{d})(z'')  dz'' \wedge d\overline{z}''$$ $$=-\int_{z' \in D_{+} \setminus \R } \sigma(z') d\tilde{u}(z') dz' \wedge d\overline{z}' -\int_{z'' \in D_{-} \setminus \R} \sigma(z'') d(\tau \circ \tilde{u} \circ \tau^{d})(z'')  dz'' \wedge d\overline{z}'' $$ 
Furthermore, explicit computation shows $$\overline{\partial}_{(J_{i} \circ u')^{d}}\tilde{u}^{d}=0$$ with $(J_{i} \circ \tilde{u})^{d}(z):=-(\tau \circ J_{i} \circ \tau) ( \tilde{u} ( \tau^{d}z))$ for $z \in \overline{D}_{-}$. In other words the expression $(J_{i} \circ \tilde{u})^{d}$ defines a continuous complex structure on the trivial bundle $\overline{D} \times \R^{2n}$ and $\tilde{u}^{d}$ is a holomorphic section. \\ \\
Let us choose for $l \geq 1$ a function $\sigma_{l} \in \mathcal{C}^{\infty}_{0}(\overline{D},\R)$ that takes values in $[0,1]$ and satisfies the following conditions: 
$$ \begin{cases} \sigma_{l} \in \mathcal{C}^{\infty}_{0}((\mu'+\frac{1-\mu'}{l})\overline{D},\R) & \\ 
\sigma_{l}= 1 & \text{on }  (\mu'+\frac{1-\mu'}{l+1})\overline{D} \\ 
\|\sigma_{l}\|_{\mathcal{C}^{1,\alpha}(\overline{D},\R)}<D_{l} &  \text{with } D_{l}=D_{l}(k,\alpha,\mu')
   \end{cases} 
   $$ 
Here we have denoted by $(\mu'+\frac{1-\mu'}{l})\overline{D}$ the closed disk of radius $\mu'+\frac{1-\mu'}{l}$ around the origin in $\C$.\\ \\
With $f_{l}:=\sigma_{l}\tilde{u}^{d} \in L^{1,2}_{0}(\overline{D},\R^{2n})$ we have
$$ \overline{\partial}f_{l}+(\overline{\partial}_{(J_{i} \circ \tilde{u})^{d}}-\overline{\partial})f_{l}=\tilde{u}^{d} \overline{\partial}_{(J_{i} \circ \tilde{u})^{d}} \sigma_{l}$$ or \begin{equation} \tag{*} (1+\frac{1}{2}((J_{i} \circ \tilde{u})^{d}-J_{st})J_{st}(T-1))\overline{\partial}f_{l}=\tilde{u}^{d} \overline{\partial}_{(J_{i} \circ \tilde{u})^{d}} \sigma_{l} \end{equation} where we applied part 2 of theorem 2.1. This equation is the principal tool that allows to obtain the desired estimates. \\ \\ Let us use $(*)$ to establish the claim of step 2. Put $l=1$. Note that for $\|(J_{i} \circ \tilde{u})^{d}-J_{st}\|_{L^{\infty}}$ small enough the operator $$1+\frac{1}{2}((J_{i} \circ \tilde{u})^{d}-J_{st})J_{st}(T-1): L^{p}(\overline{D},\R^{2n}) \rightarrow L^{p}(\overline{D},\R^{2n})$$ is invertible. Since $\tilde{u}^{d} \overline{\partial}_{(J_{i} \circ u')^{d}} \sigma_{1} \in L^{p}(\overline{D},\R^{2n})$ we deduce $\overline{\partial}f_{1} \in L^{p}(\overline{D},\R^{2n})$ and with the help of part 3 of theorem 2.1 $$ \| \overline{\partial}f_{1} \|_{L^{p}(\overline{D},\R^{2n})}< (1+\frac{1}{2} \varepsilon(1+\|T\|_{L(L^{p}(\overline{D},\R^{2n}),L^{p}(\overline{D},\R^{2n}))})) \| \tilde{u}^{d} \overline{\partial}_{(J_{i} \circ \tilde{u})^{d}}\|_{L^{p}(\overline{D},\R^{2n})}$$ $$\leq  \frac{3}{2} D_{1} \pi^{\frac{1}{p}} \|\tilde{u}^{d}\|_{L^{\infty}(\overline{D},\R^{2n})}$$ for $$\varepsilon<(1+\|T\|_{L(L^{p}(\overline{D},\R^{2n}),L^{p}(\overline{D},\R^{2n}))})^{-1}$$ Note that so far we have only used the fact that the almost complex structure $(J_{i} \circ u')^{d}$ is essentially bounded. 
\\ \\ Applying part 1 and part 4 of theorem 2.1 we obtain $$\|\tilde{u}\|_{\mathcal{C}^{1-\frac{2}{p}}(\mu'\overline{D}_{+},\R^{2n})} \leq \|\tilde{u}^{d}\|_{\mathcal{C}^{1-\frac{2}{p}}(\mu' \overline{D},\R^{2n})} \leq \|f_{1}\|_{\mathcal{C}^{1-\frac{2}{p}}(\overline {D},\R^{2n})} = \|P(\overline{\partial}f_{1})\|_{\mathcal{C}^{1-\frac{2}{p}}(\overline {D},\R^{2n})}$$ $$ \leq \|P\|_{L(L^{p}(\overline{D},\R^{2n}),\mathcal{C}^{1-\frac{2}{p}}(\overline{D},\R^{2n}))} \| \overline{\partial}f_{1}\|_{\mathcal{C}^{1-\frac{2}{p}}(\overline {D},\R^{2n})} $$ $$ < \frac{3}{2} D_{1} \pi^{\frac{1}{p}} \|P\|_{L(L^{p}(\overline{D},\R^{2n}),\mathcal{C}^{1-\frac{2}{p}}(\overline{D},\R^{2n}))}  \|\tilde{u}^{d}\|_{L^{\infty}(\overline{D},\R^{2n})} $$ $$\leq 3 D_{1} \pi^{\frac{1}{p}} \|P\|_{L(L^{p}(\overline{D},\R^{2n}),\mathcal{C}^{1-\frac{2}{p}}(\overline{D},\R^{2n}))}  \|\tilde{u}\|_{L^{\infty}(\overline{D},\R^{2n})}  =: C'_{0} \|\tilde{u}\|_{L^{\infty}(\overline{D}_{+},\R^{2n})}$$ This completes the proof of step 2. \\ \\
$Step$ 3. The mapping  $u$ is of class $\mathcal{C}^{1,\alpha'}$ up to the boundary for every $\alpha' \in (0,\alpha)$ and we have the estimate
$$ \|u \|_{\mathcal{C}^{1,\alpha'}(B_{\mu r}(z),M)}<C_{1} r^{-1-\alpha'} \text{diam} (u(B_{r}(z)))$$ with $C_{1}=C_{1}(J,X,\mu,\alpha')$. \\
 \\
We employ the fact that since $\tilde{u}$ is $\mathcal{C}^{\beta}$-continuos up to the boundary of $D_{+}$ for each $\beta \in (0,1)$ and because of the choice of $\Phi_{i}$ the complex structure $(J_{i} \circ \tilde{u})^{d}$ is of class $\mathcal{C}^{\alpha'}$ for each $\alpha' \in (0,\alpha)$ and we have the estimate $$\| (J_{i} \circ \tilde{u})^{d}-J_{st}\|_{\mathcal{C}^{\alpha'}(\frac{1+\mu'}{2}\overline{D})}<\varepsilon C'_{0} \|\tilde{u}\|_{L^{\infty}(\overline{D}_{+},\R^{2n})}<\varepsilon C'_{0} C' \delta_{\partial} $$ where the constant $C'_{0}$ is as in the previous step for $1-2/p > \alpha'/\alpha$. Hence arguing as before and using part 5 of theorem 2.1 we obtain for $f_{2}:=\sigma_{2}f$
$$ \| \overline{\partial}f_{2} \|_{\mathcal{C}^{\alpha'}(\overline{D},\R^{2n})}< 
\frac{3}{2} D_{2} C(p) \pi^{\frac{1}{p}} \|\tilde{u}^{d}\|_{L^{\infty}(\overline{D},\R^{2n})}=:C'_{1} \|\tilde{u}^{d}\|_{L^{\infty}(\overline{D},\R^{2n})}$$  Here we have assumed $$\varepsilon<(C'_{0} C' \delta_{\partial} (1+\|T\|_{L(\mathcal{C}^{\alpha'}(\overline{D},\R^{2n}),\mathcal{C}^{\alpha'}(\overline{D},\R^{2n}))}))^{-1}$$ and the constant $C(p)$ bounds the Sobolev embedding $$L^{1,p}(\overline{D},\R^{2n}) \hookrightarrow \mathcal{C}^{1-2/p}(\overline{D},\R^{2n})$$ Applying part 1 and part 5 of theorem 2.1 the assertion follows. \\ \\ 
$Step$ 4. $u$ is of class $\mathcal{C}^{1,\alpha}$ up to the boundary and 
$$ \|u \|_{\mathcal{C}^{1,\alpha}(B_{\mu r}(z),M)}<C_{2} r^{-1-\alpha} \text{diam} (u(B_{r}(z)))$$ holds with $C_{2}=C_{2}(J,X,\mu,\alpha)$. \\ \\
According to the previous step we have the estimate $$\| (J_{i} \circ \tilde{u})^{d}-J_{st}\|_{\mathcal{C}^{\alpha}(\frac{1+2\mu'}{3}\overline{D})}< \varepsilon C'_{1} C' \delta_{\partial}$$ hence we can again apply part 1 and part 5 of theorem 2.1. \\ \\
$Step$ 5. Under the assumption $k > 0$ we have $u \in L^{2,2}_{loc}(S,M)$ and the estimate $$\|u \|_{L^{2,2}(B_{\mu r}(z),M)}<C_{4} r^{-2} \text{diam} (u(B_{r}(z))) $$ with $C_{3}=C_{3}(J,X,\mu)$. \\ \\
For the proof let us again look at the equation $(*)$ from step 2: $$(1+\frac{1}{2}((J_{i} \circ \tilde{u})^{d}-J_{st})J_{st}(T-1))\overline{\partial}f_{l}=\tilde{u}^{d} \overline{\partial}_{(J_{i} \circ \tilde{u})^{d}} \sigma_{l}$$ The right hand side lies in $L^{1,2}(\overline{D},\R^{2n})$ and by the previous step we can estimate its norm in terms of the diameter of the image $\tilde{u}$. The composition $J_{i} \circ \tilde{u}$ lies in $L^{1,2}(\overline{D}_{+})$, hence using Stokes theorem and continuity of $(J_{i} \circ \tilde{u})^{d}$ as in step 2 we obtain $(J_{i} \circ \tilde{u})^{d} \in L^{1,2}(\overline{D})$ and an estimate for $\| (J_{i} \circ \tilde{u})^{d}-J_{st}\|_{L^{1,2}(\overline{D})}$. Applying Schwarz inequality and part 6 of theorem 2.1 we can argue as in the proof of step 2.  \\ \\
$Step$ 6. As preparation for the proof of higher order estimates we discuss the construction of almost complex structures on spaces of jets of pseudoholomorphic maps with the help of minimal connections.\\ \\
Let us recall the general idea. Given two almost complex manifolds $(N,j)$ and $(M,J)$, consider the vector bundle $$\pi^{1}: J^{1}(N,M)= \text{Hom}_{\C}(TN \otimes_{j} \C, TM \otimes_{J} \C) \rightarrow N \times M$$ whose fibre at $z \times x \in N \times M$ is the space of complex linear maps $(T_{z}N,j_{z}) \rightarrow (T_{x}M,J_{x})$. We want to construct an almost complex structure $\widetilde{J}^{1}$ on the total space $J^{1}(N,M)$ such that for any pseudoholomorphic $u: (N,j) \rightarrow (M,J)$ the map $J^{1}(u): N \rightarrow J^{1}(N,M)$ given by the derivative $z \mapsto du_{z} \in J^{1}(N,M)_{(z , u(z))}$ becomes pseudoholomorphic. A natural way of constructing $\widetilde{J}^{1}$ is to use minimal connections on $N$ and on $M$ (\cite{Ga}, paragraph A2). \\ \\
Let us carry out a similar construction for jets of sections of a complex vector bundle $(E,J)$ over an almost complex manifold $(N,j)$. Recall that unless we consider the special case of holomorphic bundles there is a priori no well-defined notion of a holomorphic section of $E$. However, if a connection $\nabla$ on $E$ is chosen then we can lift $J$ to an almost complex structure $\widetilde{J}$ on the total space $E$. To do this we note that $\nabla$ induces a splitting $$T_{\xi}E=  H^{\nabla}_{\xi}E \oplus V_{\xi}E$$ for $\xi \in E$, with $V_{\xi}E=\text{ker}(d\pi_{\xi}) \simeq E_{\pi(\xi)}$ and such that $d\pi_{\xi}: H^{\nabla}_{\xi}E \simeq T_{\pi(\xi)}N$. Given $X \in T_{z}N$ and $\xi \in E_{z}$ we will denote by $X^{H}_{\xi} \in  H^{\nabla}_{\xi}E$ the horizontal lift of $X$ to $T_{\xi}E$.\\ \\ For $V=(\xi,X) \in E_{\pi(\xi)} \oplus  T_{\pi(\xi)}N \simeq T_{\xi}E$ we set $$\widetilde{J}V:=(J\xi,jX) $$  Then a section $u$ of $E$ is called holomorphic if it is $j$-$\widetilde{J}$-pseudoholomorphic when viewed as a map $du: N \rightarrow E$. If $$\pi_{\xi}^{\nabla}: T_{\xi}E \rightarrow V_{\xi}E \simeq E_{\pi(\xi)}$$ denotes the projection to the vertical subspace along $H^{\nabla}_{\xi}E$ then $u$ being holomorphic is equivalent to $$\pi_{u(z)}^{\nabla} \circ du: (T_{z}N,j_{z}) \rightarrow (E_{z},J_{z})$$ being complex linear. \\ \\
Next, define the space of 1-jets of holomorphic sections of $E$ as the bundle $$ J^{1}(E)=\text{Hom}_{\C}(TN \otimes_{j} \C, TE \otimes_{\widetilde{J}} \C) \rightarrow E$$ whose fiber at $\xi \in E$ is the space of complex linear maps from $T_{\pi(\xi)}N$ to $T_{\xi}E$. Denote by $\pi^{1}$ the projection $J^{1}(E) \rightarrow E$. \\ \\ Composition from the left with $\widetilde{J}$ gives a complex structure $J^{1}$ on the bundle $J^{1}(E)$, furthermore given a holomorphic section $u$ of $E$, the differential $du$ defines a section $J^1(u)$ of $J^{1}(E)$ along the graph of $u$ by setting $J^1(u)(X)=du(X)$. We want to interpret $J^1(u)$ as a holomorphic section with respect to a suitable connection on the bundle $J^{1}(E)$.\\ \\
Let us assume that a $j$-linear connection $\nabla^{N}$ on $N$ and a $\widetilde{J}$-linear connection $\nabla^{E}$ on the total space of $E$ are chosen. 
Then 
 define a connection $\nabla^{1}$ on the bundle $J^{1}(E)$ by setting $$(\nabla^{1}_{V}L)(Y)=\nabla^{E}_{V}(LY)-L\nabla^{N}_{\pi_{*}V}Y$$ where $L$ is a local section of $J^{1}(E)$ and $V$ a vector field on the total space $E$.  
As above we can use $\nabla^{1}$ to lift $J^{1}$ to an almost complex structure $\widetilde{J}^{1}$ on the total space $J^{1}(E)$. Then arguing similarly as in Proposition A.2.10 of \cite{Ga} there is a specific choice of $\nabla^{N}$ and of $\nabla^{E}$ for which $J^{1}(u)$, viewed as a map $N \rightarrow J^{1}(E)$, will automatically be $j $-$\widetilde{J}^{1}$-holomorphic. Namely if $$N(V,W)=[\widetilde{J}V,\widetilde{J}W]-\widetilde{J}[\widetilde{J}V,W]-\widetilde{J}[V,\widetilde{J}W]-[V,W]$$ denotes the Nijenhuis tensor of $\widetilde{J}$ and $$T(V,W)=\nabla^{E}_{V}W-\nabla^{E}_{W}V-[V,W]$$ the torsion of $\nabla^{E}$ then $\nabla^{E}$ is called minimal with respect to $\widetilde{J}$ if $T=\frac{1}{4}N$ holds. Connections of this type always exist by \cite{Li}, theorem 5.4.110 on page 190, see also lemma 2.4 below. Let us reassure ourselves that if both $\nabla^{N}$ and $\nabla^{E}$ are minimal then $J^{1}(u)$ will be a pseudoholomorphic map. Given $z \in N$ and $X \in T_{z}N$, the connection $\nabla^{1}$ gives a decomposition of $d J^{1}(u)(X)$ into a horizontal and a vertical part that looks as follows: $$d J^{1}(u)(X)=(\nabla^{1}_{du(X)}du,du(X)) \in (J^{1}(E))_{u(z)} \oplus T_{u(z)}E$$  
and hence $J^{1}(u)$ being holomorphic is equivalent to the condition $$\nabla^{1}_{du(j X)}du=J \circ \nabla^{1}_{du(X)}du$$ or using the definition of $\nabla^{1}$ and complex linearity of $\nabla^{N}$ and of $\nabla^{E}$ to $$ \nabla^{E}_{du(jX)}(du(Y))-\nabla^{E}_{du(X)}(du(jY))=du(\nabla^{N}_{jX}Y-\nabla^{N}_{X}(jY))$$ 
Now $$\nabla^{N}_{jX}Y-\nabla^{N}_{X}(jY)=T^{\nabla,N}(jX,Y)+j\nabla^{N}_{Y}X +[jX,Y]-j\nabla^{N}_{X}(Y)$$ $$=T^{\nabla,N}(X,jY)-jT^{\nabla,N}(X,Y)+[jX,Y]-j[X,Y]$$
Recall that by minimality we have $T^{\nabla,N}=\frac{1}{4}N^{j}$ and the result follows from the fact that the Nijenhuis tensor is preserved under pseudoholomorphic maps, i. e. we have $N(du \cdot, du \cdot)=duN(\cdot,\cdot)$.\\ \\
To complete the construction we need to interpret $J^{1}(u)$ as a holomorphic section of some bundle with base $N$. 
Let us define the complex bundle $E'$ as $$\pi': E'= \text{Hom}_{\C}(TN \otimes_{j} \C,E \otimes_{J} \C) \rightarrow N$$ Note that since for each $\xi \in E$ we have $V_{\xi}E \simeq E_{\pi(\xi)}$ there is an isomorphism $$\text{Hom}_{\C}(TN \otimes_{j} \C, VE \otimes_{\widetilde{J}} \C) \simeq \pi^{*}E'$$ $$=E' \times_{\pi} E=\{(L',\xi) \in E' \times E \text{ s. t. } \pi'(L')=\pi(\xi)\}$$ and hence we obtain a homomorphism of complex vector bundles $J^{1}(\pi): J^{1}(E) \rightarrow E'$ defined by $(J^{1}(\pi))(L)Y=\pi^{\nabla}(LY)$. Under $J^{1}(\pi)$ the section $J^{1}(u)$ projects to a section of $E'$ which we call $u'$. If the connection $\nabla^{1}$ also projects to some connection $\nabla'$ on the bundle $E'$ then $u'$ is holomorphic with respect to $\nabla'$ and the construction is complete: having started with a holomorphic section $u$ of $(E,J,\nabla)$ we arrive at the holomorphic section $u'$ of $(E',J',\nabla')$. However, this will not be the case in general and hence we need to impose an additional condition on our choice of the connection $\nabla^{E}$ on the total space $E$. This condition is as follows: let us call $\nabla^{E}$ horizontal with respect to $\nabla$ if $\nabla^{E}$ preserves the splitting $TE= H^{\nabla}E \oplus VE$ induced by $\nabla$, i. e. given $V \in TE$ and a section $W$ of $H^{\nabla}E$ (resp. of $VE$) then $\nabla^{E}_{V}W \in H^{\nabla}E$ (resp. $\nabla^{E}_{V}W \in VE$). If this holds, then $\nabla^{E}$ induces a connection $\nabla'$ on $E'$: given $L' \in L$ define the horizontal subspace of $T_{L'}E'$ by $$H^{\nabla'}_{L'}E'= d J^{1}(\pi)(H^{\nabla^{1}}_{L} J^{1}(E))$$ where $L \in J^{1}(E)$ with $J^{1}(\pi)(L)=L'$. One checks that if $\nabla^{E}$ is horizontal then the right hand side does not depend on the choice of $L$. \\ \\ 
$Step$ 7. We establish the full statement of the theorem.\\ \\
Let us see how the discussion of step 6 applies in the setting considered in steps 2 to 5. Here we have $N=D_{+}$ and we are given an isomorphism of $E$ to the trivial bundle $D_{+} \times \R^{2n}$ endowed with the complex structure $J_{i} \circ \tilde{u}$ of class $\mathcal{C}^{1,\alpha}$, $k>0$ which we will denote by $J$. Furthermore, we will need to consider the general situation where the connection $\nabla$ on $E$ is of class $\mathcal{C}^{\alpha}$ and differs from the standard connection $\nabla_{0}$ on the trivial bundle $D_{+} \times \R^{2n}$. Let us assume that the difference $A=\nabla-\nabla_{0} \in \Gamma(\text{Hom}(T D_{+} \otimes E ,E)$ satisfies the bound $\|A\|_{\mathcal{C}^{\alpha}}<C^{\nabla}$ with some $C^{\nabla}>0$.\\ \\ 
In this situation the assertion of step 4 can be restated as follows: \\ \\
For each $\mu' \in (0,1)$ there exists a constant $\varepsilon=\varepsilon(\alpha,\mu')$ such that under the assumptions $\| J-J_{st} \|_{\mathcal{C^{\alpha}}(D_{+})}<\varepsilon$ and $J|_{D_{+} \cap \R}=J_{st}$ every section $\tilde{u} \in \mathcal{C}^{0} \cap L^{1,2}_{loc}(D_{+},\R^{2n})$ of $E$ which is holomorphic with respect to $J$ and $\nabla$ and satisfies $u'(D_{+} \cap \R) \subset \R^{m}$ lies in $\mathcal{C}^{1,\alpha}(D_{+})$ and we have the estimate $$\|\tilde{u}\|_{\mathcal{C}^{1,\alpha}(\mu'D_{+},\R^{2m})} < \widetilde{C} \|\tilde{u}\|_{L^{\infty}(D_{+},\R^{2m})}$$ with some constant $\widetilde{C}=\widetilde{C}(\alpha,\mu',C^{\nabla})$. \\ \\ 
Indeed, equation $(*)$ now reads $$ (1+\frac{1}{2}((J_{i} \circ \tilde{u})^{d}-J_{st})J_{st}(T-1))\overline{\partial}f_{l}=(\overline{\partial}_{(J_{i} \circ \tilde{u})^{d}} \sigma_{l}-\sigma_{l} (A^{0,1} \circ \tilde{u})^{d})\tilde{u}^{d}$$ where $A^{0,1}$
is the $(0,1)$-part of $A$ and $(A^{0,1} \circ \tilde{u})^{d}$ denotes the double of $A^{0,1} \circ \tilde{u}$, defined by $(A^{0,1} \circ \tilde{u})^{d}(z)=-(\tau \circ A^{0,1} \circ \tau)(\tilde{u}(\tau^{d}z))$ for $z \in D_{-}$. Note that since $\nabla$ is $J$-linear $(A^{0,1} \circ \tilde{u})^{d}$ is continuos. Hence we can repeat the argumentation of step 4 to establish the claimed estimate.\\ \\
Now let $J'$ be the complex structure on the bundle $E'$ as above. The fibre $E'_{z}$ of $E'$ at $z \in D_{+}$ can be identified with $\text{Hom}_{\C}(\C, \R^{2n} \otimes_{J'_{z}} \C)$ which gives a trivialization $E' \simeq D_{+} \times \R^{2n}$.  We notice that the boundary condition $J'|_{D_{+} \cap \R} = J_{st}$ holds by definition of $J'$. Furthermore, denoting by $\nabla^{E}_{0}$ the euclidean connection on the total space $E$ and using lemma 2.4 below, we can assume that the norm $\|\nabla^{E}-\nabla^{E}_{0}\|_{\mathcal{C}^{\alpha}}$ can be estimated above in terms of $\|J-J_{st}\|_{\mathcal{C}^{1,\alpha}}$ and of $\|\nabla-\nabla_{0}\|_{\mathcal{C}^{1,\alpha}}$. Using the definition of $\nabla'$ we obtain a bound $$\|\nabla'-\nabla'_{0}\|_{\mathcal{C}^{\alpha}}<C^{\nabla'}=C^{\nabla'}(\|J-J_{st}\|_{\mathcal{C}^{1,\alpha}},C^{\nabla})$$ here $\nabla'_{0}$ is the connection on the bundle $E'$ induced by the trivialization $E' \simeq D_{+} \times \R^{2n}$. Note that if we identify $\text{Hom}_{\C}(\C, \C^{n} )$ with $(\R^{2n},J_{st})$ then the totally real subspace $\{L \in \text{Hom}_{\C}(\C, \C^{n} ) \text{ s. t. } L(\R) \subset \R^{n}\}$ is identified with $\R^{n}$ and the boundary condition $\tilde{u}(D_{+} \cap \R) \subset \R^{n}$ implies $\tilde{u}'(D_{+} \cap \R) \subset \R^{n}$. Furthermore since $k>0$, according to step 5 we can assume $\tilde{u} \in L^{2,2}(\overline{D},R^{2n})$ and hence $\tilde{u}' \in L^{1,2}(\overline{D},R^{2m})$. \\ \\
We are now ready to apply the above formulation of the result of step 4 to the section $\tilde{u}'$ of $E'$ to obtain the estimate $$\| \tilde{u}'\|_{\mathcal{C}^{1,\alpha}((\mu')^{2} D_{+},\R^{2m})}<\widetilde{C}\|\tilde{u}'\|_{L^{\infty}(\mu'D_{+},\R^{2m})} \leq \widetilde{C}\|\tilde{u}\|_{\mathcal{C}^{1}(\mu' D_{+},\R^{2n})}$$ Now using $$\| \tilde{u} \|_{\mathcal{C}^{2,\alpha}((\mu')^{2} D_{+},\R^{2n})} \leq \| \tilde{u}'\|_{\mathcal{C}^{1,\alpha}((\mu')^{2} D_{+},\R^{2m})}+\| \tilde{u} \|_{L^{\infty}((\mu')^{2} D_{+},\R^{2n})}$$ and applying step 4 again we arrive at $$\| \tilde{u}\|_{\mathcal{C}^{2,\alpha}((\mu')^{2} D_{+},\R^{2n})} \leq \widetilde{C}\|\widetilde{u}\|_{\mathcal{C}^{1}(\mu' D_{+},\R^{2n})}+\| \tilde{u} \|_{L^{\infty}((\mu')^{2} D_{+},\R^{2n})}$$ $$ \leq (\widetilde{C}^{2}+1) \|\tilde{u}\|_{L^{\infty}(\mu' D_{+},\R^{2n})}$$
with $\tilde{C}=\tilde{C}(\alpha,(\mu')^{2},C^{\nabla'})$. \\ \\ Applying the same construction to $\tilde{u}'$ and arguing inductively gives the statement of the theorem.
\end{proof}
In course of the proof we used the following existence result:
\begin{lemma}
Let $\pi: (E,J) \rightarrow N$ be a complex vector bundle over an almost complex manifold $(N,j)$ and $\nabla$ a $J$-linear connection on $E$. Denote by $\widetilde{J}$ the almost complex structure on the total space $E$ obtained from $j$ and $J$ using $\nabla$. Then there exists a connection $\nabla^{E}$ on the total space $E$ which is minimal with respect to $\widetilde{J}$ and horizontal with respect to $\nabla$.\\ \\
Let $J$ and $\nabla$ be of differentiabilty class $\mathcal{C}^{k,\alpha}$, $k,\alpha>0$, $N$ be the half-disk $D_{+}$ and assume that $E$ is trivial as a real vector bundle, i. e. there is an isomorphism $E \simeq D_{+} \times \R^{2n}$. Denote by $\nabla_{0}$ the standard connection on the        trivial bundle $D_{+} \times \R^{2n} \rightarrow D_{+}$ and by $\nabla^{E}_{0}$ the euclidean connection on the total space $D_{+} \times \R^{2n}$. Then $\nabla^{E}$ can be chosen of class $\mathcal{C}^{k-1,\alpha}$ and such that we have an estimate  $$\|\nabla^{E}-\nabla^{E}_{0}\|_{\mathcal{C}^{k-1,\alpha}}<C(\|\nabla-\nabla_{0}\|_{\mathcal{C}^{k,\alpha}},\|J-J_{st}\|_{\mathcal{C}^{k,\alpha}})$$
\end{lemma}
\begin{proof}
Recall that the connection $\nabla$ induces a splitting $TE=H^{\nabla}E \oplus VE$ of the tangent bundle of $E$ and we have $H^{\nabla}E \simeq \pi^{*}TN$ and $VE \simeq \pi^{*}E$. If $\nabla^{N}$ is a connection on $N$ then we can define a horizontal connection $\nabla^{E}$ on the total space $E$ as $\nabla^{E}=\pi^{*}\nabla^{N} \oplus \pi^{*} \nabla$. \\ \\
It is well-known that given any connection on an almost complex manifold it is possible to add some explicitly known linear term to obtain a complex linear minimal connection, see for example \cite{Ok}, chapter 5, theorem 2.4.1. More precisely, if we define $A,B \in \text{Hom}(TE \otimes TE, TE)$ by $$A(X,Y)=\frac{1}{2}\widetilde{J}(\nabla^{E}_{X}\widetilde{J})(Y)$$ and $$B(X,Y)=\frac{1}{4}(\widetilde{J}(\nabla^{E}_{Y}\widetilde{J})(X)+(\nabla^{E}_{\widetilde{J}Y}\widetilde{J})(X))$$ then the sum $$\nabla^{E}(X,Y)+A(X,Y)+B(X,Y)$$ defines a complex linear minimal connection.  Note that by definition of $\tilde{J}$ this connection will also be horizontal with respect to $\nabla$.\\ \\
The second part of the lemma follows directly from the above formulae.
\end{proof}
We can summarize the results of theorem 2.2 and theorem 2.3 as follows:
\begin{corollary}
Under the assumptions of theorem 2.3 there are $\delta_{S}=\delta_{S}(J,\mu)>0$ and $C_{S}=C_{S}(J,\mu)>0$ such that for any nonconstant pseudoholomorphic map $u \in \mathcal{C}^{0} \cap L^{1,2}_{loc} (S,M)$ with $u(\partial S) \subset X$, any $z \in S$ and any $r>0$ such that $ \text{diam}(u(B_{r}(z)))<\delta_{S}$ we have 
$$\|u \|_{\mathcal{C}^{k+1,\alpha}(B_{\mu r}(z),M)}<C_{S} r^{-(k+1+\alpha)} \text{diam} (u(B_{r}(z)))$$
\end{corollary}
\begin{proof}
Choose $\delta_{S}< \text{min}\{\delta, \mu \delta_{\partial}/2\}$. For each $z \in S$ we either have $B(z,r) \subset B(z',2 r)$ and hence also $B(z,r/ \mu) \subset B(z',2r / \mu)$ with some $z' \in \partial S$ and theorem 2.3 applies on $B(z', 2r / \mu)$ or alternatively $d(z, \partial S) > r$, therefore $B(z,r) \cap \partial S = \emptyset$ and theorem 2.2 can be used.
\end{proof}
Let us briefly recall two other results that will be relevant: the monotonicity lemma and the theorem about the removal of point singularities.
\begin{lemma}
Let the almost complex structure $J$ be of class $\mathcal{C}^{\alpha}$ for some $\alpha>0$. There exist constants $r_{M}=r_{M}(J)>0$, $\delta_{M}=\delta_{M}(J)>0$ and $C_{M}=C_{M}(J)>0$ with the following properties:
\begin{enumerate}
 \item For any $r<r_{M}$, any compact connected Riemann surface $\Sigma$ with boundary, any pseudoholomorphic map $u: \Sigma \rightarrow M$ and any $x \in u(\Sigma)$ with $u(\partial \Sigma) \cap B(x,r) \subset X$ we have the estimate $$\mathcal{E}(u|_{u^{-1}(B(x,r))}) \geq C_{M} r^{2}$$
\item The inequality $\text{diam}(u(\Sigma)) < 2 \delta_{M}$ implies that $u$ is a constant map.
\end{enumerate}
\end{lemma}
Here the energy $\mathcal{E}(u)$ of a pseudoholomorphic map $u: \Sigma \rightarrow M$ is defined as $\mathcal{E}(u)=\|du\|^{2}_{L^{2}(\Sigma,M)}$ and $B(x,r) \subset M$ denotes the metric ball of radius $r$ around $x$.
\begin{proof}
 This follows from \cite{Si}, propositions 4.3.1 and 4.7.2.
\end{proof}
Let us again denote by $D_{+}=D \cap \{z \in \C: \text{ s. t. Im}(z) \geq 0 \}$ the upper half-disk.
\begin{theorem}
Let $J$ be of class $\mathcal{C}^{\alpha}$, $\alpha>0$. 
\begin{enumerate} 
 \item Any pseudoholomorphic map $u: D \setminus {0} \rightarrow M$ with $\mathcal{E}(u) < \infty$ extends continuosly to $D$.
\item Any pseudoholomorphic map $u: D_{+} \setminus {0} \rightarrow M$ with $u(\R \cap D_{+} \setminus {0}) \subset X$ and $\mathcal{E}(u) < \infty$ extends continuosly to $D_{+}$.
\end{enumerate}
\end{theorem}
\begin{proof}
Part 1 is theorem 4.5.1 and part 2 theorem 4.7.3 in \cite{Si}.
\end{proof}

\section{Geometry of hyperbolic surfaces}
An important source of geometric intuition that leads to the compactness theorem is the geometry of hyperbolic surfaces. Let us state the results that will be useful in course of the proof of theorem 1.1. \\ \\
Recall that a hyperbolic surface is a smooth surface not of exceptional type endowed with a hyperbolic structure, i. e. a complete Riemannian metric of constant sectional curvature $-1$.
\begin{theorem}
Let $\{S_{n}\}_{n \in \N}$ be a sequence of hyperbolic surfaces such that the sequence $\{\mathcal{A}(S_{n})\}$ of areas is bounded by some constant $C < \infty$. Then there exist a subsequence, again denoted by $\{S_{n}\}_{n \in \N}$, a smooth surface $S$, a regular collection $\Delta$ of curves on $S$, a hyperbolic structure $h_{\infty}$ on the complement $S \setminus \Delta$ and diffeomorphisms $\psi_{n}: S \rightarrow S_{n}$, such that convergence $\psi^{*}_{n} h_{n} \rightarrow h_{\infty}$ of hyperbolic structures holds in $\mathcal{C}^{\infty}_{loc}(S \setminus \Delta)$. If we denote by $\Sigma_{n}$ resp. by $\Sigma$ the smooth surface obtained from $S_{n}$ resp. from $S$ by removal of all punctures then $\psi_{n}$ extend to diffeomorphisms $\varphi_{n}: \Sigma \rightarrow \Sigma_{n}$. Furthermore $\psi_{n}$ is an isometry in a neighbourhood of $\Sigma \setminus S$.
\end{theorem}
Here, we call a subset $\Delta \subset S$ a regular collection of curves if the connected components of $\Delta$ are boundary components of $S$, simple closed curves on $S \setminus \partial S$ and simple paths with endpoints on $\partial S$ that are transversal to $\partial S$. 
\begin{proof}
A proof of the theorem under the additional assumption that the lenghts of all boundary components of the surfaces $\{S_{n}\}_{n \in \N}$ be bounded above can be found in \cite{Hu}, paragraph IV.5. The proof relies on the existence of decompositions of the surfaces $\{S_{n}\}_{n \in \N}$ into hyperbolic pairs of pants with uniformly bounded lengths of boundary components, see theorem IV.3.7 in \cite{Hu}. The general case of the theorem can be reduced to this situation as follows: Consider the sequence of Shottky doubles $\{S^{d}_{n}\}_{n \in \N}$. Using lemma 5.10 in \cite{Iv} it is possible to obtain suitable pair of pants decompositions of the surfaces $\{S^{d}_{n}\}_{n \in \N}$ which are in addition symmetric with respect to the natural antiholomorphic involutions. We thus obtain the claim arguing as in the proof of theorem 3.7. 
\end{proof}
Let us also recall the thick-thin decomposition.    
\begin{theorem}
Let $S$ be a hypebolic surface without boundary and $0<\rho'<\rho<\text{arsinh} (1)$. Every connected component of $$\{z  \in S: \rho'<r_{inj}(S,z)<\rho\}$$ is isometric to either
\begin{enumerate}
\item $$ U_{\rho',\rho,l}:= \{z\in \H : 2\rho'<d(z,e^{l}z) < 2 \rho \}/(z \sim e^{l}z)$$ for some $l<2\text{arsinh} (1)$ or to
\item $$ V_{\rho',\rho}:= \{ z \in \H: \frac{1}{2 sinh \rho'}> Im(z) > \frac{1}{2 sinh \rho} \}/(z \sim z+1)$$ 
\end{enumerate}
Furthermore $U_{\rho',\rho,l}$ and $V_{\rho}$ are annuli of modulus $$mod(U_{\rho',\rho,l}) \geq \pi(\frac{1}{\sinh \rho'}-\frac{1}{\sinh \rho})$$ and $$mod(V_{\rho',\rho})= \pi(\frac{1}{\sinh \rho'}-\frac{1}{\sinh \rho})$$ 
\end{theorem}
\begin{proof}
This follows from proposition IV.4.2 with the help of examples I.5.5 and I.5.6 in \cite{Hu}.
\end{proof}
\section{Proof of theorem 1.1}
The idea of proof is to set punctures on the Riemann surfaces $\{\Sigma_{n}\}_{n \in \N}$ to obtain a sequence $\{S_{n}\}_{n \in \N}$ of hyperbolic surfaces to which theorem 3.1 can be applied. The crucial step is to construct the limit map $u_{\infty}$ by showing that these punctures can be chosen so that the gradient of $u_{n}$ is uniformly bounded with respect to the hyperbolic metric on $S_{n}$ and hence the theorem of Arzel\`a-Ascoli applies. This relies on the thick-thin decomposition and can be formulated as follows in the case of curves with boundary:
\begin{lemma}
Let $S$ be a hyperbolic surface with boundary and $u: S \rightarrow M$ a pseudoholomorphic map with $u(\partial \Sigma) \subset X$. For each $\varepsilon>0$ there exists a constant $\rho=\rho(\varepsilon,J,\mathcal{E}(u))>0$ such that for every $z \in S$ the metric ball $B(z,\rho) \subset S$ is contained in some connected subset $A(z) \subset S$ which satisfies one of the following two conditions:
\begin{enumerate}
 \item The complement $\partial A(z) \setminus \partial S$ has two connected components $\partial_{1}A(z)$ and $\partial_{2}A(z)$ which are both smooth curves satisfying $$l(u|_{\partial_{i}A(z)}) < \varepsilon$$
 \item $A(z) \subset S$ is an embedded annulus with boundary components satisfying $\partial_{1}A(z) \subset \partial S$ and $$ l(u|_{\partial_{2}A(z)}) < \varepsilon $$
\end{enumerate}

\end{lemma}
\begin{proof}
We first recall that to construct the subset $A(z)$ it is sufficient to find adjacent annuli with large moduli. More precisely we claim that it suffices to show that there exists $\rho=\rho(\varepsilon,J,\mathcal{E}(u))>0$ such that for each $z \in S$ one of the following holds:
\begin{enumerate}
\item There are numbers $r_{1},r_{2},r_{3}>0$ with $$r_{i} \geq \frac{2 \pi}{\mathcal{E}^{2}(u)}$$ for $i=1,3$ and an embedded annulus $S^{1} \times (0,r_{1}+r_{2}+r_{3}) \subset S^{d}$ with $B(z,\rho) \subset S^{1} \times (r_{1}, r_{1}+r_{2})$ \\ \\ Furthermore, for each $t \in (0,r_{1}+r_{2}+r_{3}) $ the intersection $S^{1} \times \{t\} \cap S$ is connected. 
\item There exist $r,r'>0$ with $$r' \geq \frac{2 \pi}{\mathcal{E}^{2}(u)}$$ and an embedded annulus $S^{1} \times (0,r+r') \subset S$ with $B(z,\rho) \subset S^{1} \times (0, r)$
\end{enumerate}
Indeed, in the situation of part one we compute
$$\mathcal{E}(f) \geq \mathcal{E}(f|_{(S^{1}\times(0,r_{1})) \cap S}) =  \int \limits_{t \in [0,r_{1}])} \int \limits_{s \in \text{I}_{t}} \|df(s,t)\|^{2} ds dt$$ $$ \geq \frac{1}{2\pi} \int \limits_{t \in [0,r_{1}]} (\int \limits_{s \in \text{I}_{t}} \|df(s,t)\| ds)^{2} dt \geq \frac{1}{2\pi} \int \limits_{t \in [0,r_{1}]} l^{2}(f|_{\text{I}_{t}\times\{t\}}) dt$$ where we have put $(S^{1}\times(0,r_{1})) \cap S=: \cup_{t \in (0,r_{1})} \text{I}_{t} \times \{t\}$. Hence there exists $t_{1} \in (0,r_{1})$ with $$l^{2}(f|_{\text{I}_{t_{1}}\times\{t_{1}\}}) \leq \frac{2\pi \mathcal{E}(f)}{r_{1}} \leq \varepsilon^{2}$$ By the same argument we find $t_{2} \in (r_{1}+r_{2},r_{1}+r_{2}+r_{3})$ with $$l(f|_{\text{I}_{t_{2}}\times\{t_{2}\}}) \leq \varepsilon$$ Therefore $A(z):=S \cap (S^{1} \times (t_{1},t_{2}))$ satisfies condition 1 of the claim of the theorem. Similarly in the situation of part two we construct $A(z) \subset S$ satisfying the other condition. \\ \\
For the following, let us fix an arbitrary $\rho>0$. Consider two cases:\\ \\
$Case$ 1. The injectivity radius of $S^{d}$ at $z$ satisfies the estimate $r_{inj}(S^{d},z) \geq 4\rho^{1/2}$. \\ \\
Here we can use the argument from the proof of theorem V.2.3 in \cite{Hu}. To this effect we assume $3 \rho < \rho^{1/2}$, choose $z' \in S$ with $d(z,z')=2 \rho$ and consider the following adjacent annuli in $S^{d}$: $$A_{1}=B(z', \rho^{1/2}) \setminus B(z', 3\rho) $$  $$A_{2}=B(z', 3 \rho) \setminus B(z', \rho) $$ and $$A_{3}=B(z',  \rho) \setminus \{z'\}$$ We have $\text{mod}(A_{3})=\infty$, $\text{mod}(A_{1}) \rightarrow \infty$ for $\rho \rightarrow 0$ and $B(z,\rho) \subset A_{2}$. \\ \\
Note that $B(z', \rho^{1/2}) \cap \partial S$ is connected: $B(z', 3 \rho^{1/2})$ is isometric to a ball of radius $2 \rho^{1/2}$ in $\H$ and if we again denote by $\tau^{d}: S^{d} \rightarrow S^{d}$ the natural antiholomorphic involution then $B(z', \rho^{1/2}) \cup \tau^{d}(B(z', \rho^{1/2})) \subset B(z', 3 \rho^{1/2})$ is isometric to a union of two intersecting metric balls of radius $\rho^{1/2}$ in $\H$, hence simply connected.\\ \\ 
$Case$ 2. We have $r_{inj}(S^{d},z) < 4\rho^{1/2}$.\\ \\
Assuming $4\rho^{1/2}<\text{arsinh}(1)$ we can apply theorem 3.2 to deduce that either \\ \\
a) There exists a neighbourhood $U \subset S^{d}$ of $z$ which is isometric to $$\{z' \in \H \text{: } d(z',e^{l}z')<2 \text{arsinh}(1)\}/(z \sim e^{l}z)$$ 
b) A neighbourhood $V \subset S^{d}$ of $z$ is isometric to $$\{z' \in \H \text{: } \text{Im}(z)> \frac{1}{2}\}/(z \sim z+1)$$
Furthermore by theorem 3.2 we can assume that both $U$ and $V$ are invariant with respect to the involution $\tau^{d}$. This immediately implies $V \cap \partial S = \emptyset$ since otherwise $V=(V \cap S) \cup (V \cap \tau^{d}(S))$ would be a union of two pairwise isometric subsets yet the puncture of $V$ can only be contained in one of them. We see that in the case b) we can argue similarly to Case 2 (b) in the proof of lemma V.2.3 in \cite{Hu}. \\ \\
As for case b), we choose $\rho < \rho' < \text{arsinh}(1)$ and put $$A_{2}=\{z' \in U \text{: } r_{inj}(S^{d},z')< \rho' \}$$ The complement $U \setminus A_{2}$ consists of two annuli $A_{1}$ and $A_{3}$ and applying the formulae from theorem 3.2 one can see that for $\rho$ and $\rho'$ small their moduli are sufficiently large. \\\\
Finally, note that if $U \cap \partial S \neq \emptyset$ then either $U \cap \partial S$ is a connected component of $\partial S$ or $U \cap S$ and $U \cap \tau^{d} S$ are two pairwise isometric half-annuli. To see this, we note that for each $\rho'' \in (0,\text{arsinh}(1))$ the subset $$\{z' \in U \text{: } r_{inj}(S^{d},z')= \rho'' \} $$ consists of two copies of $S^{1}$ and is invariant under $\tau^{d}$. Hence either $\{z' \in U \text{: } r_{inj}(S^{d},z')= \rho'' \} \cap \partial S = \emptyset$ for each $\rho'$ or for every $\rho'$ the intersection $\{z' \in U \text{: } r_{inj}(S^{d},z')= \rho'' \} \cap \partial S$ consists of exactly four points. In the first case we arrive at part 1 of the claim of the theorem and otherwise at part 2.
\end{proof}
We can now explain the proof of theorem 1.1. Following the idea from 1.5B in \cite{Gr}, we deduce from lemma 4.1 and the monotonicity statement 2.6: \\ \\ 
There exists a constant $\delta_{0}=\delta_{0}(J)$, such that for each $\delta<\delta_{0}$ and every $n \in \N$ there is a finite subset $F_{n} \subset \Sigma \setminus \partial \Sigma$ with the following properites
\begin{enumerate}
\item We have $2<|F_{n}| \leq N $ for some constant $N$ depending only on $C_{\mathcal{E}}$, $\delta$ and $J$. We denote by $h_{n}$ the hyperbolic structure on $\Sigma \setminus F_{n}$ induced by $j_{n}$ and by $S_{n}$ the hyperbolic surface $(\Sigma \setminus F_{n},h_{n})$.
\item There exists a constant $\rho$ depending only on $C_{\mathcal{E}}$ and on $\delta$, such that for every $z \in S_{n}$ the inclusion $f(B(z,\rho)) \subset B(f(z),\delta)$ holds for the metric ball $B(z,\rho)$ around $z$ in $S_{n}$. 
\end{enumerate}
Now let us choose $\mu \in (0,1)$ and $\delta<\delta_{S}$ where $\delta_{S}$ denotes the constant from corollary 2.5. We obtain a uniform bound on $\|u_{n}\|_{\mathcal{C}^{k+1,\alpha}(S_{n},M)}$. Applying 3.1 we find a surface $S$ and diffeomorphisms $\psi_{n}: S \rightarrow S_{n}$ with $\psi^{*}_{n}h_{n} \rightarrow h_{\infty}$, here $\Delta \subset S$ is a regular collection of curves on $S$ and $h_{\infty}$ a hyperbolic structure on $S \setminus \Delta$. Using the second property of $F_{n}$, corollary 2.5 and applying the theorem of Arzel\`a-Ascoli we obtain uniform $\mathcal{C}^{k+1,\alpha'}$-convergence $u_{n} \circ \psi_{n} \rightarrow u_{\infty}$ on compact subsets of $S \setminus \Delta$, here $\alpha' \in (0,\alpha)$ and $u_{\infty}: S \rightarrow M$ is pseudoholomorphic on $S \setminus \Delta$ with respect to the complex structure $j_{\infty}$ induced by $h_{\infty}$. Note that convergence $\mathcal{E}(u_{n} \circ \psi_{n} ) \rightarrow \mathcal{E}(u_{\infty} )$ follows as in the case of closed curves from the fact that the total area of $S_{n}$ is determined by the homeomorphism type and we have a uniform bound on the gradient of $u_{n}$ with respect to $h_{n}$.  \\ \\
Lifting all punctures of $S$ we obtain $\Sigma$ together with a complex structure $j_{\infty}$ on $\Sigma \setminus \Delta$. If we denote by $F$ the complement $\Sigma \setminus S$ then, after composing with a suitable diffeomorphism supported in a neighbourhood of $F$ we may assume $\varphi_{n}$ to be an isometry near $F$, as a straightforward application of the Riemann mapping theorem shows. Furthermore, the second part of theorem 3.1 implies that $\psi_{n}$ extend to diffeomorphisms $\varphi_{n}: \Sigma \rightarrow \Sigma$. Using theorem 2.7, $u_{\infty}$ extends continuosly to the whole of $\Sigma$ and with the help of the regularity statements 2.2 and 2.3 the extension is pseudoholomorphic. As in the case of closed curves, using the monotonicity lemma 2.6 and convergence of energy one obtains $u_{n} \circ \varphi_{n} \rightarrow u_{\infty}$ pointwise on $\Sigma$. Uniform $\mathcal{C}^{k+1,\alpha'}$-convergence $u_{n} \circ \varphi_{n} \rightarrow u_{\infty}$ on compact subsets of $\Sigma \setminus \Delta=(S \setminus \Delta) \cup F$ will follow from uniform $\mathcal{C}^{0}$-convergence on $\Sigma$ since $\psi_{n}$ is isometric near $F$. Let us discuss uniform $\mathcal{C}^{0}$-convergence in some more detail as it is of independent interest.  \\ \\
 Suppose for a contradiction that there exists some point $z \in \Sigma \setminus S \cup \Delta$ and a sequence $\{z_{n}\}_{n \in \N} \subset \Sigma$ with $z_{n} \rightarrow z$ and $d(u_{n}(z_{n}),u_{\infty}(z_{n}))>\varepsilon>0$. We will asume $z \in \{ \gamma \}$ for some $\gamma \in \Delta$ as the case $z \in \Sigma \setminus S$ can be handled similarly. Fix some $\delta>0$ and a neighbourhood $U \subset S$ of $z$ such that $\text{diam}(u_{\infty}(U))<\varepsilon/2$ and $\mathcal{E}((u_{\infty})|_{U})<\delta$. Because of uniform convergence $u_{n} \circ \varphi_{n} \rightarrow u_{\infty}$ on compact subsets $S \setminus \Delta$ we have $u_{n}(\partial U) \subset B_{\varepsilon/2}(u_{\infty}(z))$ for sufficiently large $n$. Applying theorem 2.6 to a connected component of $U \setminus \{ \gamma \cup \{z_{n}\} \}$ we obtain $\mathcal{E}((u_{n})|_{U}) \geq C_{M} (\varepsilon/4)^{2}$. Now convergence of energies and uniform $\mathcal{C}^{1}$-convergence on compact subsets of $S \setminus \Delta$ implies $\mathcal{E}((u_{n})|_{U}) < 2 \delta$ for large $n$ and hence for $\delta=C_{M} \varepsilon^{2}/32$ we obtain a contradiction. The proof of theorem 1.1 is complete.  \\ \\

\vspace{2ex}
\begin{center}
Viktor Fromm, Department of Mathematical Sciences, University of Durham, Durham DH1 3LE, United Kingdom, viktor.fromm@dur.ac.uk
\end{center}
\vspace{2ex}

\end{document}